\documentclass[final,11pt,times]{elsarticle}
\usepackage{}

\usepackage{amsfonts}
\usepackage{amsmath}
\usepackage{amssymb}
\usepackage{graphics}
\usepackage{graphicx}
\usepackage{mathrsfs}
\usepackage{indentfirst}
\usepackage{stmaryrd}
\usepackage{latexsym}
\usepackage{xypic}
\usepackage[latin1]{inputenc}

\usepackage{color}

\newtheorem{definition}{Definition}[section]
\newtheorem{lemma}[definition]{Lemma}
\newtheorem{theorem}[definition]{Theorem}
\newtheorem{proposition}[definition]{Proposition}
\newtheorem{example}[definition]{Example}

\newtheorem{remark}[definition]{Remark}

\newproof{proof}{\textbf{Proof}}

\journal{.}

\begin{document}

\begin{frontmatter}



\title{\textbf{Semi-overlap Functions and Novel Fuzzy Reasoning Algorithms with Applications }}


\author{Xiaohong Zhang$^{a}$, Mei Wang$^{b,*}$, Benjam\'{i}n Bedregal$^c$, Mengyuan Li$^{a}$, Rong Liang$^{a}$ }
\cortext[cor1]{Corresponding author. \\
Email addresses: zhangxiaohong@sust.edu.cn(X.H. Zhang), wangmeimath@163.com(M. Wang),\\ bedregal@dimap.ufrn.br (B. Bedregal), 201706060120@sust.edu.cn(M.Y. Li),\\ 1137882903@qq.com(R. Liang). }

\address[A]{School of Mathematics and Data Science, Shaanxi University of Science and Technology, Xi'an 710021, Shaanxi, China}
\address[B]{School of Electrical and Control Engineering, Shaanxi University of Science and Technology, Xi'an 710021, Shaanxi, China}
\address[C]{Departamento de Inform$\acute{a}$tica e Matem$\acute{a}$tica Aplicada, Universidade Federal do Rio Grande do Norte, Campus Universit$\acute{a}$rio s/n, 59072-970 Natal, Brazil}


\begin{abstract}
It is worth noticing that a fuzzy conjunction and its corresponding fuzzy implication can form a residual pair if and only if it is left-continuous. In order to get a more general result related on residual implications that induced by aggregation functions, we relax the definition of general overlap functions, more precisely, removing its right-continuous, and then introduce a new kind of aggregation functions, which called semi-overlap functions. Subsequently, we study some of their related algebraic properties and investigate their corresponding residual implications. Moreover, serval scholars have provided kinds of methods for fuzzy modus ponens (FMP,for short) problems so far, such as Zadeh's compositional rule of inference (CRI, for short), Wang's triple I method (TIM, for short) and the quintuple implication principle (QIP, for short). Compared with CRI and TIM methods, QIP method has some advantages in solving FMP problems, in this paper, we further consider the QIP method for FMP problems and prove that it satisfies the reducibility of multiple-rules fuzzy reasoning. Finally, we propose a new classification algorithm that based on semi-overlap functions and QIP method, which called SO5I-FRC algorithm. Through the comparative tests, the average accuracy of SO5I-FRC algorithm is higher than FARC-HD algorithm. The experimental results indicate that semi-overlap functions and QIP method have certain advantages and a wide range of applications in classification problems.
\end{abstract}

\end{frontmatter}

\section{Introduction}
\label{intro}

Overlap functions were introduced in \cite{Bustince1} by Bustince, which is a special class of aggregation functions. This concept arise from some problems in information aggregation process, such as image processing \cite{Jurio}, decision making \cite{Bustince2}, community detection problems \cite{Gomez1} and classification \cite{Elkano1,Elkano2}. As a generalization of overlap functions, Miguel et al. \cite{Miguel} introduced the concept of general overlap functions. Notice that the difference between binary general overlap functions and overlap functions is boundary condition \cite{Gomez2}. In the definition of overlap functions, the boundary condition is a necessary and sufficient condition, while in the binary general overlap functions, the boundary condition is a sufficient condition. The authors in \cite{Miguel} made an experiment and showed that the general overlap function behave better when calculating the matching degree in some classification problems. There is a close relation between
fuzzy conjunction, which is a generalization form of aggregate functions, and fuzzy implication. Fuzzy implication has been widely studied in theory and practice \cite{Baczynski}. The most studied generalization is obtained by t-norm, called residual implication. However, Bustince et.al \cite{Bustince1,Bustince2} and Fodor et.al \cite{Fodor} have pointed out that the associativity of the t-norm is unnecessary in many applications,
such as decision making and classification problems. Therefore, it is necessary to construct the residual implication corresponding to the aggregation  function that does not meet the law of association. On the other hand, a fuzzy conjunction and its corresponding fuzzy implication can be form a residual pair if and only if it is left-continuous. In order to get a more general result related on residual implications that induced by aggregation functions, we relax the definition of general overlap functions, more precisely, removing its right-continuous, and then introduce a new kind of aggregation functions, which called semi-overlap functions in the present paper.

The most fundamental inference models of fuzzy reasoning are fuzzy modus ponens(FMP) and fuzzy modus tollens(FMT), which can be respectively expressed as follows:

\begin{center}
FMP: for given $A\rightarrow B$(rule) and $A^*$(input), calculate $B^*$(output)
\end{center}

\begin{center}
FMT: for given $A\rightarrow B$(rule) and $B^*$(input), calculate $A^*$(output)\\
\end{center}
where $A,A^*\in \mathcal{F}(U)$(the set of all fuzzy subsets on $U$) and $B,B^*\in \mathcal{F}(V)$(the set of all fuzzy subsets on $V$). In order to solve FMP problems,  Zadeh \cite{Zadeh} proposed compositional rule of inference (CRI, for short). Although this method is simple in calculation, it lacks clear logic sense. Therefore, the basic problem of fuzzy reasoning has attracted great attention and a series of meaningful achievements have been made. Among them, G.J. Wang \cite{Wang} proposed the full implication triple I method of fuzzy reasoning (TIM, for short).  This method effectively improved the CRI method. However, in CRI and TIM methods, the closeness of $A$ and $A^*$ (or $B$ and $B^*$) are not explicitly used in the process of calculating the consequence, which sometimes makes the computed approximation useless or misleading. In order to characterized the approximation $A^*$ and $A$ (or $B$ and $B^*$) in the process of fuzzy reasoning, B.K. Zhou et al \cite{Zhou} proposed the quintuple implication principle(QIP, for short). This method efficiently improves CRI and TIM methods. Subsequently, M.X. Luo et al \cite{Luo} investigate QIP method based on interval-valued fuzzy inference for FMP and FMT problems. In the present paper, we will continue to discuss the QIP method based on semi-overlap functions and their residual implications.

Fuzzy Rule-Based Classification Systems(FRBCSs) \cite{Ishibuchi} are one of the most popular methods in pattern recognition and machine learning.
FRBCSs is also used in many practical applications, such as image processing \cite{Nakashima} and anomaly intrusion detection \cite{Tsang}. In order to improve the interpretability of classification rules and avoid unnatural boundaries in attribute division, researchers have made different studies on classification systems based on fuzzy association rules \cite{Chen,Hua,Hu,Hua1,Pach,Yi}. In particular, J. Alcal\'{a}-Fdez \cite{AlcalaFdez} propose a fuzzy association rule-based classification method for high-dimensional problems (FARC-HD) to obtain an accurate and compact fuzzy rule-based classifier with a low computational cost. And in \cite{Elkano}, based on FARC-HD algorithm, $n$-dimensional overlap functions are used to solve the multi-class classification problems, and good results are obtained. In this paper, based on semi-overlap functions and quintuple implication principle, we propose a new classification algorithm, which is called SO5I-FRC algorithm. And the experimental results show that the average accuracy of classification using SO5I-FRC algorithm is higher than that of FARC-HD algorithm.

The paper is organized as follows. In Section 2, we recall some preliminary concepts related to this paper. In Section 3, we introduce the concept of semi-overlap functions and investigate the residual implication induced by semi-overlap functions. In Section 4, we study the quintuple implication principle and present an algorithm that satisfies the reducibility of multiple-rules fuzzy reasoning. In Section 5, we apply semi-overlap functions and quintuple implication principle to propose a new algorithm, which called
SO5I-FRC algorithm. Through comparative experiments, we conclude that the average accuracy of SO5I-FRC algorithm is higher than FARC-HD algorithm.

\section{Preliminaries}

In this part, we will recall some basic definitions and results, which will be used in the following sections.

\subsection{Aggregation functions}

\begin{definition}\emph{(\cite{Bustince1})
A function $A:[0,1]^2\rightarrow [0,1]$ is said to be an aggregation function if it satisfies the following conditions:}

\emph{$(A1)$ $A$ is increasing in each variable;}

\emph{$(A2)$ $A(0,0)=0$ and $A(1,1)=1$.}
\end{definition}

\begin{definition}\emph{(\cite{Bustince1})
A bivariate function $T:[0,1]^2\rightarrow [0,1]$ is said to be a t-norm if it satisfies the following conditions:}

\emph{$(T1)$ $T(u,v)=T(v,u)$;}

\emph{$(T2)$ $T(u,T(v,w))=T(T(u,v),w)$;}

\emph{$(T3)$ $T$ is increasing;}

\emph{$(T4)$ $T(u,1)=u$, for any $u\in[0,1]$.}\\
\emph{Moreover, a t-norm $T$ is called}

\emph{$(T5)$ continuous if it is continuous in both arguments at the same time.}

\emph{$(T6)$ positive if $T(u,v)=0$, then either $u=0$ or $v=0$.}
\end{definition}

\begin{definition}\label{2.1.3}
\emph{(\cite{Klement}) A t-norm $T$ is left-continuous if, for any  $u\in [0,1]$ and $\{v_i| i\in I\}\subseteq [0,1]$, for a non-empty set of index $I$, it satisfies:}
\begin{center}
$T(u, \sup\{v_i|i\in I\})=\sup\{T(u, v_i)|i\in I\}$.
\end{center}
\end{definition}

\begin{definition}\emph{(\cite{Bustince1,Bedregal})
A mapping $O:[0,1]^2\rightarrow [0,1]$ is called an overlap function if it satisfies the following conditions: for any $u,v\in [0,1]$,}

\emph{$(O1)$ $O(u,v)=O(v,u)$;}

\emph{$(O2)$ $O(u,v)=0$ if and only if $uv=0$;}

\emph{$(O3)$ $O(u,v)=1$ if and only if $uv=1$;}

\emph{$(O4)$ $O$ is increasing;}

\emph{$(O5)$ $O$ is continuous.}

\end{definition}

\begin{definition}\emph{(\cite{Miguel})
A $n$-ary function $GO:[0,1]^n\rightarrow [0,1]$ is called a general overlap function if the following conditions hold: for all $u_1, u_2,\cdots, u_n\in [0,1]$,}

\emph{$(GO1)$ $GO(u,v)=GO(v,u)$;}

\emph{$(GO2)$ If $\Pi^n_{i=1}u_i=0$, then $GO(u_1,u_2,\cdots,u_n)=0$;}

\emph{$(GO3)$ If $\Pi^n_{i=1}u_i=1$, then $GO(u_1,u_2,\cdots,u_n)=1$;}

\emph{$(GO4)$ $GO$ is increasing;}

\emph{$(GO5)$ $GO$ is continuous.}

\end{definition}

\begin{definition}\emph{(\cite{Paiva})
A mapping $QO:[0,1]^2\rightarrow [0,1]$ is called a quasi-overlap function if it satisfies the following conditions: for any $u,v\in [0,1]$,}

\emph{$(QO1)$ $QO(u,v)=QO(v,u)$;}

\emph{$(QO2)$ $QO(u,v)=0$ if and only if $uv=0$;}

\emph{$(QO3)$ $QO(u,v)=1$ if and only if $uv=1$;}

\emph{$(QO4)$ $QO$ is increasing.}
\end{definition}

\subsection{Fuzzy implication and fuzzy conjunction}

\begin{definition}\emph{(\cite{Krol})
A function $C:[0,1]^2\rightarrow [0,1]$ is called a fuzzy conjunction if the following conditions hold:}

\emph{$(C1)$ $C$ is increasing;}

\emph{$(C2)$ $C(1,1)=1$;}

\emph{$(C3)$ $C(0,0)=C(0,1)=C(1,0)=0$.}
\end{definition}

It is easy to know that t-norms, overlap functions, quasi-overlap functions and binary general overlap functions are fuzzy conjunctions.

\begin{definition}\emph{(\cite{Krol,Zhou1})
A function $I:[0,1]^2\rightarrow [0,1]$ is a fuzzy implication if it satisfies the following conditions: for each $u,v,w\in [0,1]$,}

\emph{$(I1)$ If $u\leq v$, then $I(v,w)\leq I(u,w)$;}

\emph{$(I2)$ If $v\leq w$, then $I(u,v)\leq I(u,w)$;}

\emph{$(I3)$ $I(0,0)=1$;}

\emph{$(I4)$ $I(1,1)=1$;}

\emph{$(I5)$ $I(1,0)=0$.}
\end{definition}

Let $C:[0,1]^2\rightarrow [0,1]$ be a fuzzy conjunction. Then we can define the function $I_C:[0,1]^2\rightarrow [0,1]$ as follows, for all $u,v\in[0,1]$,

\begin{equation}
I_C(u,v)=\sup\{w\in [0,1]|C(u,w)\leq v\}
\label{eq-3.1}
\end{equation}

\begin{theorem} \label{3.2.1}
 \emph{(\cite{Krol}) Let $C$ be a fuzzy conjunction. Then the function $I_C:[0,1]^2\rightarrow [0,1]$ given by Eq.(\ref{eq-3.1}) is a fuzzy implication if and only if $C$ fulfils the condition:}
\begin{equation}
 C(1,v)> 0, ~~\forall v\in (0,1].
\label{eq-3.2}
\end{equation}
\end{theorem}

If for any $ u, v, w\in [0,1]$, we have
\begin{center}
~~~~~~~~~~~~~~~$C(u,w)\leq v\Leftrightarrow I_C(u,v)\geq w$. ~~~~~~~~~~~~$\mathbf{(RP)}$\\
\end{center}
Then we say that the fuzzy conjunction  $C$ and the function $I_C$ defined in $Eq.(\ref{eq-3.1})$ satisfy the residuation property $\mathbf{(RP)}$.

\begin{theorem} \label{3.2.2}
\emph{ (\cite{Krol}) Let $C:[0,1]^2\rightarrow [0,1]$ be a fuzzy conjunction satisfying the condition (\ref{eq-3.2}), and $I_C: [0,1]^2\rightarrow [0,1]$ be a fuzzy implication derived from $C$. Then the following statements are equivalent:}

\emph{$(i)$ $C$ is left-continuous with respect to the second variable;}

\emph{$(ii)$ $C$ and $I_C$ satisfy the residuation property $\mathbf{(RP)}$;}

\emph{$(iii)$ $I_C(u,v)=\max\{w\in [0,1]|C(u,w)\leq v\}$, for all $u,v\in [0,1]$.}
\end{theorem}

\begin{definition}\emph{(\cite{Bedregal,Zhou1,Dimuro2})
A fuzzy implication $I:[0,1]^2\rightarrow [0,1]$ satisfies: }

\emph{$\mathbf{(NP)}$ The left neutrality property $\Leftrightarrow$ $\forall v\in[0,1]:$ $I(1,v)=v$.}

\emph{$\mathbf{(LOP)}$ The left ordering property $\Leftrightarrow$ $\forall u,v\in[0,1]:$ $u\leq v\Rightarrow I(u,v)=1$.}

\emph{$\mathbf{(ROP)}$ The right ordering property $\Leftrightarrow$ $\forall u,v\in[0,1]:$ $I(u,v)=1\Rightarrow u\leq v $.}

\emph{$\mathbf{(OP)}$ The ordering property $\Leftrightarrow$ $\forall u,v\in[0,1]$: $u\leq v\Leftrightarrow I(u,v)=1$.}

\emph{$\mathbf{(EP)}$ The exchange principle $\Leftrightarrow$ $\forall u,v,w \in[0,1]:$ $I(u,I(v,w))=I(v,I(u,w))$.}

\emph{$\mathbf{(IP)}$ The identity principle $\Leftrightarrow$ $\forall u\in[0,1]$: $I(u,u)=1$.}

\emph{$\mathbf{(CB)}$ The consequent boundary $\Leftrightarrow$ $\forall u,v\in[0,1]$: $v\leq I(u,v)$.}

\emph{$\mathbf{(SIB)}$ The sub-iterative Boolean law $\Leftrightarrow$ $\forall u,v\in[0,1]:$ $I(u,v)\leq I(u,I(u,v))$.}

\emph{$\mathbf{(IB)}$ The iterative Boolean law $\Leftrightarrow$ $\forall u,v\in[0,1]$: $I(u,v)=I(u,I(u,v))$.}

\end{definition}

\section{Semi-overlap functions}

\subsection{Semi-overlap functions}

In this part, we introduce the notion of semi-overlap functions and study some of their related algebraic properties.

\begin{definition}
\emph{A binary function $SO:[0,1]^2\rightarrow [0,1]$ is called a semi-overlap function if the following conditions hold: for any $u,v\in [0,1]$,}

\emph{$(S1)$ $SO(u,v)=SO(v,u)$;}

\emph{$(S2)$ If $uv=0$, then $SO(u,v)=0$;}

\emph{$(S3)$ If $uv=1$, then $SO(u,v)=1$;}

\emph{$(S4)$ $SO$ is increasing;}

\emph{$(S5)$ $SO$ is left-continuous.}
\end{definition}

Here, as shown in Figure \ref{fig-CF}, we discuss the relations between semi-overlap functions and other aggregation functions.

\begin{figure}[!ht]
  \centering
  \includegraphics[width=9.0cm,height=5.0cm]{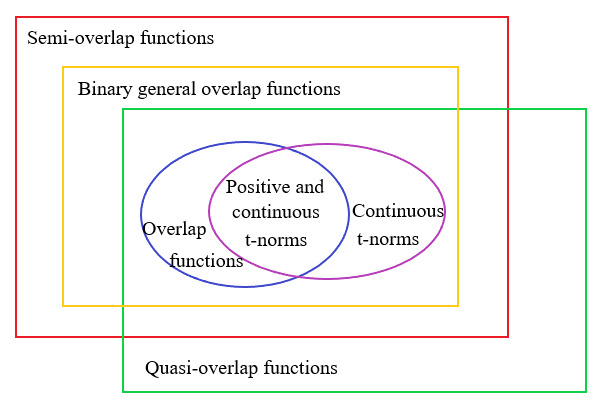}\\
  \caption{The included relation of different aggregation functions\label{fig-CF}}
\end{figure}

\begin{remark}
\emph{(1) Semi-overlap functions are bivariate general overlap functions which consider left continuity instead of continuity. So semi-overlap functions are generalization of bivariate general overlap functions.}

\emph{(2) The bivariate general overlap function is generalization of the overlap function. Hence, by (1), we can get that the semi-overlap function is also generalization of the overlap function.}

\emph{(3) There is no inclusive relations between semi-overlap functions and quasi-overlap functions. Since the latter does not require continuity, the  former is left-continuous. On the other hand, quasi-overlap functions has no zero-divisor, but semi-overlap functions may have zero-divisor.}
\end{remark}

Then we will give some examples of semi-overlap functions.

\begin{example}
\label{ex-2}
\emph{$(1)$. For a fix $a\in (0,1]$ and $u,v\in[0,1]$, the function}

\begin{center}
$SO(u,v)$=$\left\{
  \begin{array}{ll}
    0, & if~u+v\leq a; \\
    u\wedge v, & if~u+v> a.
  \end{array}
\right.$
\end{center}
\emph{is a semi-overlap function, but it is not a quasi-overlap function.}

\emph{$(2)$ The function, defined for any $u,v\in[0,1]$ by}

\begin{center}
$SO(u,v)$=$\left\{
  \begin{array}{ll}
    0, & if~u+v\leq 1; \\
    uv, & otherwise.
  \end{array}
\right.$\\
\end{center}
\emph{is a semi-overlap function. However, it is not  a binary general overlap function.}

\emph{$(3)$ The function $SO: [0,1]^2\rightarrow [0,1]$ given by:}

\begin{center}
$SO(u,v)=\min\{u,v\}\max\{u^2,v^2\}$\\
\end{center}
\emph{is a non-associative semi-overlap function having $1$ as neutral element.}

\emph{$(4)$ The function given by $SO_p(u,v)=u^pv^p$ with $p> 0$ and $p\neq 1$ is a semi-overlap function, but it is neither associative nor have $1$ as neutral element.}
\end{example}

\begin{definition}
\emph{A semi-overlap function $SO:[0,1]^2\rightarrow [0,1]$ is called deflationary semi-overlap function if}

\emph{$(S6)$  $\forall u\in [0,1]$: $SO(u,1)\leq u$,}\\
\emph{and $SO$ is called inflationary semi-overlap function if}

\emph{$(S7)$  $\forall u\in [0,1]$: $SO(u,1)\geq u$.}
\end{definition}

\begin{example}
\emph{Let us consider the function given in Example 3.3(4). Then $SO_p$ is deflationary when $p> 1$ and inflationary when $0< p< 1$.}
\end{example}

\begin{theorem}
\emph{The function $SO:[0,1]^2\rightarrow [0,1]$ is a semi-overlap function if and only if there are two function $f, g:[0,1]^2\rightarrow [0,1]$ such that}
\begin{center}
$SO(u,v)=\frac{f(u,v)}{f(u,v)+g(u,v)}$,\\
\end{center}
\emph{and satisfies the following conditions:}

\emph{$(1)$ $f(u,v)+g(u,v)\neq 0$ for all $u,v\in [0,1]$.}

\emph{$(2)$ $f$ and $g$ are commutative;}

\emph{$(3)$ If $uv=0$, then $f(u,v)=0$;}

\emph{$(4)$ If $uv=1$, then $g(u,v)=0$;}

\emph{$(5)$ $f$ is increasing and $g$ is decreasing;}

\emph{$(6)$ $f$ is left-continuous and $f+g$ is right-continuous;}
\end{theorem}

\begin{proof}
$(\Rightarrow)$ Let $SO$ be a semi-overlap function. Now, suppose that $f(u,v)=SO(u,v)$ and $g(u,v)=1-SO(u,v)$, then $f(u,v)+g(u,v)=1\neq 0$. Therefore, the function

\begin{center}
$SO(u,v)=\frac{f(u,v)}{f(u,v)+g(u,v)}$\\
\end{center}
can be defined. And it is not hard to prove that the conditions $(1)-(6)$ hold.

$(\Leftarrow)$ Suppose that $f,g:[0,1]^2\rightarrow [0,1]$ satisfies the conditions $(1)-(6)$, we will show that
\begin{center}
$SO(u,v)=\frac{f(u,v)}{f(u,v)+g(u,v)}$\\
\end{center}
is a semi-overlap function. Obviously, it satisfies $(S1)$, $(S2)$ and $(S3)$. Next, we will  prove that it  satisfies the conditions $(S4)$ and $(S5)$.

$(S4)$ If $u_1\leq u_2$, since $f$ is increasing and $g$ is decreasing, then $f(u_1,v)\leq f(u_2,v)$ and $g(u_2,v)\leq g(u_1,v)$. Moreover, we can get
\begin{center}
$f(u_1,v)g(u_2,v)\leq f(u_2,v)g(u_1,v)$\\
\end{center}
and
\begin{center}
~$f(u_1,v)f(u_2,v)+f(u_1,v)g(u_2,v)\leq f(u_1,v)f(u_2,v)+f(u_2,v)g(u_1,v)$.\\
\end{center}
Hence
\begin{eqnarray*}
SO(u_1,v)&=&\frac{f(u_1,v)}{f(u_1,v)+g(u_1,v)}\\&\leq&\frac{f(u_2,v)}{f(u_2,v)+g(u_2,v)}
\\&=&SO(u_2,v).
\end{eqnarray*}

$(S5)$ Since $f+g$ is right-continuous, we have $\frac{1}{f+g}$ is left-continuous. Then according to the left continuity of $f$, we can get $\frac{f}{f+g}$ is left-continuous.
\end{proof}

\begin{theorem}
\emph{Let $SO_1, SO_2,\cdots, SO_m$ be semi-overlap functions and $r_1,r_2,\cdots,r_m$ be nonnegative weights with $\sum\limits_{j=1}^m r_j=1$. Then $SO(u,v)=\sum\limits_{j=1}^m r_j SO_j(u,v)$ is also a semi-overlap function.}
\end{theorem}

\begin{proof}
Conditions $(S1)-(S4)$ are easy proved. Next, we prove that $SO$ satisfies $(S5)$. According to Definition \ref{2.1.3}, if $SO$ is left-continuous, then for any $u\in [0,1]$ and any $\{v_i|i\in I\}\subseteq [0,1]$, it follows that:
\begin{center}
$SO(u,\sup\{v_i|i\in I\})=\sup\{SO(u,v_i)|i\in I\}$.\\
\end{center}
Hence, we have
\begin{eqnarray*} SO(u,\sup\limits_{i\in I}\,\, v_i)&=& \sum\limits_{j=1}^m r_j SO_j(u,\sup\limits_{i\in I}\,\, v_i)\\&=& \sum\limits_{j=1}^m r_j(\sup\limits_{i\in I}\,\,SO_j(u,v_i))\\&=& \sup\limits_{i\in I}\sum\limits_{j=1}^m r_j SO_j(u,v_i)\\&=& \sup\limits_{i\in I}\,\, SO(u,v_i).
\end{eqnarray*}
\end{proof}

\subsection{Residual implications derived from semi-overlap functions}

Let $SO:[0,1]^2\rightarrow [0,1]$ be a semi-overlap functions. Then we defined $I_{SO}:[0,1]^2\rightarrow [0,1]$ as follows:

\begin{center}
$I_{SO}(u,v)=\sup\{w\in [0,1]|SO(u,w)\leq v\}$, ~~~~ for all $u,v\in [0,1]$.\\
\end{center}

\begin{lemma}
\label{3.2.3}
\emph{If $SO$ is an inflationary semi-overlap function, then it is a fuzzy conjunction that satisfies Eq.(\ref{eq-3.2}).}
\end{lemma}

\begin{proof}
Obviously, every semi-overlap function is a fuzzy conjunction. Now, if $SO$ is inflationary, then by $(S7)$, for any $v\in(0,1]$, we can get $SO(1,v)\geq v> 0$.
\end{proof}

From Theorem \ref{3.2.1}, Theorem \ref{3.2.2} and Lemma \ref{3.2.3}, we can immediately draw the following conclusion.

\begin{proposition}
\emph{If $SO$ is an inflationary semi-overlap function, then the function $I_{SO}$ that induced by semi-overlap function is a residual implication. Moreover, $SO$ and $I_{SO}$ satisfies $\mathbf{(RP)}$.}
\end{proposition}

In Table 1, we give some examples of semi-overlap functions and their residual implication.

\begin{flushleft}
\begin{tabular}{l|l}
  \hline
  $Semi-overlap~ functions$ & $Residual~ implications$ \\
  \hline
  (1)~$ SO(u,v)=\min\{u^p, v^p\}$ & $I_{SO}(u,v)=\left\{
  \begin{array}{ll}
  1, &  ~~~~~~~~~~~~~~~~~~~u^p\leq v; \\
  \sqrt[p]{v} , & ~~~~~~~~~~~~~~~~~~~u^p> v .
  \end{array}
  \right.$ \\
  \hline
  (2)~$SO(u,v)=\left\{
  \begin{array}{ll}
  0, &  u+v\leq a;a\in[0,1]\\
  min\{u,v\} , &u+v>a .
  \end{array}
  \right.$ & $I_{SO}(u,v)=\left\{
  \begin{array}{ll}
  1, &u\leq v; \\
   max\{(a-u), v\} , &u> v .
   \end{array}
   \right.$ \\
   \hline
  (3)~$SO(u,v)=\left\{
  \begin{array}{ll}
   0, &  ~~~~~~~~~~u+v\leq 1; \\
   uv , & ~~~~~~~~~~otherwise .
   \end{array}
   \right.$ & $I_{SO}(u,v)=\left\{
  \begin{array}{ll}
  1, & u\leq v; \\
   max\{(1-u), \frac{v}{u}\} , &otherwise  .
   \end{array}
   \right.$ \\
   \hline
  (4)~$SO(u,v)=\left\{
  \begin{array}{ll}
   0, &  ~~~~~~~~u+v= 0; \\
   \frac{2uv}{u+v} , & ~~~~~~~~otherwise .
   \end{array}
   \right.$ & $I_{SO}(u,v)=\left\{
  \begin{array}{ll}
  1, & ~~~~u=0; \\
   max\{0, \frac{uv}{2u-v}\} , & ~~~~otherwise .
   \end{array}
   \right.$  \\
   \hline
   (5)~$SO(u,v)=\max\{0,u+v-1\}$ & $ I_{SO}(u,v)=\left\{
  \begin{array}{ll}
  1, & ~~~~~~~~~u\leq v; \\
   v-u+1 , &~~~~~~~~~otherwise  .
   \end{array}
   \right.$ \\
  \hline
  (6)~$SO(u,v)=\max\{\frac{u+v-1}{u+v-uv},0\}$ & $ I_{SO}(u,v)=\left\{
  \begin{array}{ll}
  1, & ~~~~~~~~~~~~~u\leq v; \\
   \frac{uv+1-u}{uv+1-v} , &~~~~~~~~~~~~~otherwise.
   \end{array}
   \right.$ \\
   \hline
  (7)~$SO(u,v)=uv$ & $ I_{SO}(u,v)=\left\{
  \begin{array}{ll}
  1, & ~~~~~~~~~~~~~~~~~~~~~u\leq v; \\
   \frac{v}{u} , &~~~~~~~~~~~~~~~~~~~~~u> v  .
   \end{array}
   \right.$ \\
   \hline
  (8)~$SO(u,v)=\frac{u+v}{2}$ &  $I_{SO}(u,v)=\left\{
  \begin{array}{ll}
  1, & ~~~~~~~~~~~~~u\leq 2v-1; \\
   2v-u , & ~~~~~~~~~~~~~otherwise  .
   \end{array}
   \right.$ \\
   \hline
  (9)~$SO(u,v)=\min\{u,v\}$ & $ I_{SO}(u,v)=\left\{
  \begin{array}{ll}
  1, & ~~~~~~~~~~~~~~~~~~~~~u\leq v; \\
   v , &~~~~~~~~~~~~~~~~~~~~~otherwise  .
   \end{array}
   \right.$ \\
   \hline

\end{tabular}
\end{flushleft}
\begin{center}
Table 1. Semi-overlap functions and their residual implications.
\end{center}

\begin{proposition}\label{prop.2}
\emph{Let $SO:[0,1]^2\rightarrow [0,1]$ be a semi-overlap function and $I_{SO}$ be a residual implication. Then it holds that:}

\emph{(1) $SO(u,v)=\min\{w\in[0,1] | I_{SO}(u,w)\geq v\}$.}

\emph{(2) $I_{SO}$ satisfies $\mathbf{(NP)}$ if and only if $1$ is the neutral element of $SO$.}

\emph{(3) $I_{SO}$ satisfies $\mathbf{(EP)}$ if and only if $SO$ is associative.}

\emph{(4) $I_{SO}$ satisfies $\mathbf{(IP)}$ if and only if $SO$ satisfies $(S6)$.}

\emph{(5) $I_{SO}$ satisfies $\mathbf{(LOP)}$ if and only if $SO$ satisfies $(S6)$.}

\emph{(6) $I_{SO}$ satisfies $\mathbf{(ROP)}$ if and only if $SO$ satisfies $(S7)$.}

\emph{(7) $I_{SO}$ satisfies $\mathbf{(OP)}$ if and only if $1$ is the neutral element of $SO$.}

\emph{(8) $I_{SO}$ satisfies $\mathbf{(CB)}$ if and only if $SO(u,v)\leq \min\{u,v\}$.}

\emph{(9) If $I_{SO}$ satisfies $\mathbf{(CB)}$, then $I_{SO}$ satisfies $\mathbf{(SIB)}$.}

\emph{(10) If $SO(u,v)=\min\{u,v\}$, then $I_{SO}$ satisfies $\mathbf{(IB)}$.}

\emph{(11) If $SO$ has $1$ as neutral element, then $I_{SO}$ satisfies $\mathbf{(CB)}$.}
\end{proposition}

\begin{proof}
(1) Straightforward from Theorem 13 in \cite{Krol}.

(2) $(\Rightarrow)$~Suppose that $I_{SO}$ satisfies $\mathbf{(NP)}$, that is, for any $v\in[0,1]$,
\begin{equation}
\label{3.2.4}
I_{SO}(1,v)=\max\{t\in [0,1]|SO(1,t)\leq v\}=v.\\
\end{equation}
It follows that $SO(1,v)\leq v$. Assume that there exists a $v_0\in [0,1]$ such that $SO(1,v_0)< v_0$. Then there also exists $w< v_0$ such that $SO(1,v_0)\leq w$. According to $\mathbf{(RP)}$, we have $w< v_0\leq I_{SO}(1,w)$, which is contradiction to the Eq. (\ref{3.2.4}). Hence $SO(1,v)=v$.

$(\Leftarrow)$~ Suppose that for any $v\in[0,1]$, $SO(1,v)=v$. Then we have $I_{SO}(1,v)=\max\{t\in [0,1]|SO(1,t)\leq v\}=\max\{t\in [0,1]|t\leq v\}=v$. Hence $I_{SO}$ satisfies $\mathbf{(NP)}$.

(3) Let $u,v,w\in [0,1]$.
 \begin{eqnarray*}(\Rightarrow)~~ SO(u,SO(v,w))&=& \min\{t\in[0,1] | I_{SO}(u,t)\geq SO(v,w)\}~~~~~~(item ~~(1))\\&= & \min\{t\in[0,1] | I_{SO}(w,I_{SO}(u,t))\geq v\}~~~~~~(by~~ (RP))\\&=& \min \{t\in[0,1] | I_{SO}(u,I_{SO}(w,t))\geq v\}~~~~~~(by~~ (EP))\\&=& \min\{t\in[0,1] | I_{SO}(w,t)\geq SO(u,v)\}~~~~~~(by ~~(RP))\\&=& SO(w,SO(u,v))~~~~~~~~~~~~~~~~~~~~~~~~~~~~~~~~~~~~(item~~ (1))\\&=& SO(SO(u,v),w)~~~~~~~~~~~~~~~~~~~~~~~~~~~~~~~~~~~~(by~~ commutative).
\end{eqnarray*}
\begin{eqnarray*}(\Leftarrow)~~I_{SO}(u,I_{SO}(v,w))&=&\max\{t\in[0,1] | SO(u,t)\leq I_{SO}(v,w)\}\\&=&\max\{t\in[0,1] | SO(v,SO(u,t))\leq w\}~~~~~~(by~~ (RP))\\&=&\max\{t\in[0,1] | SO(u,SO(v,t))\leq w\}~~~~~~(by~~ associative)\\&=&\max\{t\in[0,1] | SO(v,t)\leq I_{SO}(u,w)\}~~~~~~(by~~ (RP))\\&=&I_{SO}(v,I_{SO}(u,w)).
\end{eqnarray*}

(4) For each $ u\in[0,1]$, $I_{SO}(u,u)=\max\{t\in[0,1]|SO(u,t)\leq u\}=1$ if and only if $SO(u,1)\leq u$.

(5) $(\Rightarrow)$~ If $u\leq v$, then $I_{SO}(u,v)=1$. In particular, $I_{SO}(u,u)=1$. Hence, by (4), we have $SO$ satisfies $(S6)$.

$(\Leftarrow)$~  If $u\leq v$, then by $(S6)$, we get $SO(u,1)\leq u\leq v$. Thus, $I_{SO}(u,v)=\max\{w\in [0,1]|SO(u,w)\leq v\}=1$.

(6) $(\Rightarrow)$~ By item (1), we have $SO(u,1)=\min\{w\in[0,1]| I_{SO}(u,w)\geq 1\}$, then $I_{SO}(u,SO(u,1))=1$. By $\mathbf{(ROP)}$,  $I_{SO}(u,v)=1\Rightarrow u\leq v$, hence, we can get $SO(u,1)\geq u$.

$(\Leftarrow)$~ If $SO(u,1)\geq u$ for each $u\in [0,1]$, then
\begin{eqnarray*} I_{SO}(u,v)=1&\Rightarrow&\max\{w\in[0,1] | SO(u,w)\leq v\}=1\\&\Rightarrow& SO(u,1)\leq v\\&\Rightarrow& u\leq v.
\end{eqnarray*}

(7) Obviously, it can be obtained from (5) and (6).

(8)$(\Rightarrow)$ Suppose that $I_{SO}$ satisfies $\mathbf{(CB)}$. For any $u,v\in[0,1]$, we have
\begin{eqnarray*}
v \leq I_{SO}(u,v) &\Rightarrow& v\leq \max\{w\in [0,1]|SO(u,w)\leq v\}\\&\Rightarrow& v\in \max\{w\in [0,1]|SO(u,w)\leq v\}\\&\Rightarrow& SO(u,v)\leq v.
\end{eqnarray*}
That is
\begin{equation}
v \leq I_{SO}(u,v)\Rightarrow SO(u,v)\leq v.
\end{equation}
Analogously, we have
\begin{equation}
u\leq I_{SO}(v,u)\Rightarrow u\leq SO(v,u).
\end{equation}
Since $SO$ is commutative, by Eqs.(4) and (5), we can concludes that
\begin{center}
$v\leq I_{SO}(u,v)\Rightarrow SO(u,v)\leq \min\{u,v\}$.
\end{center}

$(\Leftarrow)$ For all $u,v\in[0,1]$, it holds that
\begin{eqnarray*}SO(u,v)\leq \min\{u,v\}&\Rightarrow& SO(u,v)\leq v\\&\Rightarrow& v\in \{w\in[0,1]|SO(u,w)\leq v\}\\&\Rightarrow& v\leq \max \{w\in[0,1]|SO(u,w)\leq v\}\\&\Rightarrow& v\leq I_{SO}(u,v)
\end{eqnarray*}

(9) Straightforward from (I2).

(10) For any $u,v\in[0,1]$, if $SO(u,v)=\min\{u,v\}$, then it holds that

\begin{center}
$ I_{SO}(u,v)=\max\{w\in [0,1]|\min\{u,w\}\leq v\}
=\left\{
  \begin{array}{ll}
  1, &u\leq v; \\
  v, &u> v .
   \end{array}
   \right.$\\
\end{center}
If $u\leq v$, then $I_{SO}(u,I_{SO}(u,v))=I_{SO}(u,1)=1=I_{SO}(u,v)$. On the other hand, if $u> v$, then $I_{SO}(u,I_{SO}(u,v))=I_{SO}(u,v)$.

(11) Assume that for any $v\in [0,1]$, $SO(1,v)=v$. Since $SO$ is increasing, $SO(u,v)\leq SO(1,v)$, that is $SO(u,v)\leq v$. According to $\mathbf{(RP)}$, we have $v\leq I_{SO}(u,v)$.
\end{proof}

\section{Quintuple implication principle}

\subsection{Quintuple implication principle of FMP}

In this subsection, we will use the quintuple implication principle (QIP, for short), which induced by semi-overlap functions and their residual implications, to solve FMP problems. The FMP problems as follows:

\begin{center}
\begin{tabular}{l}
  rule~~~~ $A\rightarrow B$ \\

  input~~~$A^*$ \\
   \hline
  output~~~~~~~~~$B^*$.\\
\end{tabular}
\end{center}

\begin{definition}
\emph{Let $U$ and $V$ be two non-empty sets, $A,A^*\in \mathcal{F}(U)$, $B\in \mathcal{F}(V)$. If $B^*$ in FMP is the minimum fuzzy set of $\mathcal{F}(V)$ such that, for any $u\in U, v\in V$,}

\begin{center}
$(A(u)\rightarrow B(v))\rightarrow((A^*(u)\rightarrow A(u))\rightarrow (A^*(u)\rightarrow B^*(v)))$\\
\end{center}
\emph{has the maximum value, i.e. $1$, then $B^*$ is called the FMP-solution of quintuple implication principle.}

\end{definition}

\begin{theorem}\label{4.1.1}
 \emph{Let the implication $\rightarrow$ in FMP be a residual implication derived from the semi-overlap function $\odot$ with 1 as neutral element. Then the FMP-solution $B^*$ of quintuple implication principle is calculated as follows:}
\begin{center}
$B^*(v)=\sup\limits_{u\in U}\{A^*(u)\odot[(A^*(u)\rightarrow A(u))\odot(A(u)\rightarrow B(v))]\}$,\\
\end{center}
\emph{for any $v\in V$.}
\end{theorem}

\begin{proof}
(1) First, we shall prove that, for each $u\in U$, $v\in V$,

\begin{center}
$(A(u)\rightarrow B(v))\rightarrow[(A^*(u)\rightarrow A(u))\rightarrow (A^*(u)\rightarrow B^*(v))]=1$.\\
\end{center}
In particular, by $B^*(v)=\sup\limits_{u\in U}\{A^*(u)\odot[(A^*(u)\rightarrow A(u))\odot(A(u)\rightarrow B(v))]\}$, we have
\begin{center}
$A^*(u)\odot[(A^*(u)\rightarrow A(u))\odot(A(u)\rightarrow B(v))]\leq B^*(v)$.\\
\end{center}
Further, by $\mathbf{(RP)}$, it follows that
\begin{center}
$(A^*(u)\rightarrow A(u))\odot(A(u)\rightarrow B(v))\leq A^*(u)\rightarrow B^*(v)$.\\
\end{center}
Using $\mathbf{(RP)}$ again, we get
\begin{center}
$(A(u)\rightarrow B(v))\leq(A^*(u)\rightarrow A(u))\rightarrow (A^*(u)\rightarrow B^*(v))$.\\
\end{center}
Since the semi-overlap function has 1 as the neutral element, according to Proposition \ref{prop.2}(7), we have
\begin{center}
$(A(u)\rightarrow B(v))\rightarrow[(A^*(u)\rightarrow A(u))\rightarrow (A^*(u)\rightarrow B^*(v))]=1$.
\end{center}

(2) Then, we prove that $B^*(v)$ is the minimum in $\mathcal{F}(V)$. Let $C^*(v)\in \mathcal{F}(V)$ such that
\begin{center}
$(A(u)\rightarrow B(v))\rightarrow[(A^*(u)\rightarrow A(u))\rightarrow (A^*(u)\rightarrow C^*(v))]=1$.\\
\end{center}
By Proposition \ref{prop.2}(7), we have
\begin{center}
$A(u)\rightarrow B(v)\leq (A^*(u)\rightarrow A(u))\rightarrow (A^*(u)\rightarrow C^*(v))$.\\
\end{center}
And by using $\mathbf{(RP)}$ twice , we can get
\begin{center}
$A^*(u)\odot[(A(u)\rightarrow B(v))\odot (A^*(u)\rightarrow A(u))]\leq  C^*(v)$,\\
\end{center}
which implies that $C^*(v)$ is an upper bound of the set
\begin{center}
$\{A^*(u)\odot[(A^*(u)\rightarrow A(u))\odot(A(u)\rightarrow B(v))]|u\in U, v\in V\}$\\
\end{center}
that is, $ B^*(v)\leq C^*(v)$. Therefore, $B^*$ is the FMP-solution of quintuple implication principle.
\end{proof}

\begin{example}
\label{4.1.2}
\emph{Let us consider the following rule:}

\begin{center}
\emph{rule: If $u$ is $A$, then $v$ is $B$}
\end{center}
\begin{flushleft}
~~~~~~~~~~~~~~~~~~~~~~~~~~~~~~~~~~~~~~~~~~~~~~~~~~~~\emph{input: $A^*$}
\end{flushleft}
\begin{flushleft}
~~~~~~~~~~~~~~~~~~~~~~~~~~~~~~~~~~~~~~~~~~~~~~~~~~~~\emph{output: $B^*$.}
\end{flushleft}
\emph{where $A, A^*\in \mathcal{F}(U)$ with $U=\{u_1, u_2, u_3, u_4\}$ and $B, B^*\in \mathcal{F}(V)$ with $V=\{v_1, v_2, v_3, v_4\}$. These fuzzy sets are given by:}
\begin{center}
$A=\{(u_1, 0), (u_2, 0.4), (u_3, 0.7), (u_4, 1)\}$,
\end{center}
\begin{center}
$~B=\{(v_1, 0.2), (v_2, 0.5), (v_3, 0.9), (v_4, 1)\}$,
\end{center}
\begin{center}
$~~~~A^*=\{(u_1, 0.1), (u_2, 0.3), (u_3, 0.5), (u_4, 0.9)\}$.
\end{center}
\emph{The semi-overlap function we choose and its corresponding residual implication are}
\begin{center}
$u\odot v=min\{u,v\}=u\wedge v$\\
\end{center}
\emph{and}

\begin{center}
$u\rightarrow v$=$\left\{
  \begin{array}{ll}
    1, & if~u\leq v; \\
    v, & if~u> v.
  \end{array}
\right.$
\end{center}
\emph{Then we can calculate}

  \begin{eqnarray*}(1)~~~A^*(u)\rightarrow A(u)&=&[0.1\rightarrow 0 ~~0.3\rightarrow 0.4 ~~0.5\rightarrow 0.7 ~~0.9\rightarrow 1]\\&=&[0~~1~~1~~1].
\end{eqnarray*}
 \begin{eqnarray*}(2)~~~~~A(u)\rightarrow B(v)&=&\left[ {\begin{array}{*{20}{c}}
   0\rightarrow 0.2 & 0\rightarrow 0.5 & 0\rightarrow 0.9 & 0\rightarrow 1 \\
   0.4\rightarrow 0.2 & 0.4\rightarrow 0.5 & 0.4\rightarrow 0.9 & 0.4\rightarrow 1  \\
   0.7\rightarrow 0.2 & 0.7\rightarrow 0.5 & 0.7\rightarrow 0.9 & 0.7\rightarrow 1  \\
   1\rightarrow 0.2 & 1\rightarrow 0.5  & 1\rightarrow 0.9 & 1\rightarrow 1   \\
\end{array}} \right]\\&=&\left[ {\begin{array}{*{20}{c}}
   1 & 1 & 1 & 1 \\
   0.2 & 1 & 1 & 1  \\
   0.2 & 0.5 & 1 & 1  \\
   0.2 & 0.5  & 0.9 & 1   \\
\end{array}} \right].
\end{eqnarray*}
 \begin{eqnarray*}(3)~~~~~(A^*(u)\rightarrow A(u))\odot (A(u)\rightarrow B(v))&=&[0~~1~~1~~1]\wedge \left[ {\begin{array}{*{20}{c}}
   1 & 1 & 1 & 1 \\
   0.2 & 1 & 1 & 1  \\
   0.2 & 0.5 & 1 & 1  \\
   0.2 & 0.5  & 0.9 & 1   \\
\end{array}} \right]\\&=& \left[ {\begin{array}{*{20}{c}}
   0 & 0 & 0 & 0 \\
   0.2 & 1 & 1 & 1  \\
   0.2 & 0.5 & 1 & 1  \\
   0.2 & 0.5  & 0.9 & 1   \\
\end{array}} \right].
\end{eqnarray*}
\emph{Further, we can get}
\begin{eqnarray*}
B^*(v)&=& \sup\limits_{u\in U}\{A^*(u)\odot[(A^*(u)\rightarrow A(u))\odot(A(u)\rightarrow B(v))]\}\\&=&\sup\limits_{u\in U}\left\{[0.1~~ 0.3~~ 0.5~~ 0.9]\wedge \left[ {\begin{array}{*{20}{c}}
   0 & 0 & 0 & 0 \\
   0.2 & 1 & 1 & 1  \\
   0.2 & 0.5 & 1 & 1  \\
   0.2 & 0.5  & 0.9 & 1   \\
\end{array}} \right]\right\}\\&=&\sup\limits_{u\in U}\left\{\left[ {\begin{array}{*{20}{c}}
   0 & 0 & 0 & 0 \\
   0.2 & 0.3 & 0.3 & 0.3  \\
   0.2 & 0.5 & 0.5 & 0.5 \\
   0.2 & 0.5  & 0.9 & 0.9   \\
\end{array}} \right]\right\}\\&=&[0.2~~0.5~~0.9~~0.9].
\end{eqnarray*}
\emph{Hence, by Theorem \ref{4.1.1}, we have}
\begin{center}
$B^*(v)=\{(v_1, 0.2), (v_2, 0.5), (v_3, 0.9), (v_4, 0.9)\}$.
\end{center}
\end{example}

Next, we study the reducibility of QIP method of FMP problems. For FMP problems, reversibility means that if $A^*$ equals $A$, we hope that $B^*$ is equal to $B$.

\begin{theorem}
\label{4.1.3}
\emph{Let the implication $\rightarrow$ in FMP be a residual implication derived from the semi-overlap function $\odot$ with 1 as neutral element. If $A$ is a normal fuzzy set, that is  $\exists u_0\in U$ such that $A(u_0)=1$, then the FMP problem satisfies the reversibility of QIP method.}
\end{theorem}

\begin{proof}
According to Theorem \ref{4.1.1}, FMP-solution of QIP method is the following formula:

\begin{center}
$B^*(v)=\sup\limits_{u\in U}\{A^*(u)\odot[(A^*(u)\rightarrow A(u))\odot(A(u)\rightarrow B(v))]\}$. \\
\end{center}
Suppose that $A^*=A$ is a normal fuzzy set, then we will prove that $B^*=B$. Since $\odot$ has 1 as neutral element and  $A(u)\rightarrow A(u)=1$, it follows that

\begin{center}
$(A(u)\rightarrow A(u))\odot(A(u)\rightarrow B(v))\leq A(u)\rightarrow B(v)$.\\
\end{center}
If $A^*=A$, then
\begin{center}
$(A^*(u)\rightarrow A(u))\odot(A(u)\rightarrow B(v))\leq A^*(u)\rightarrow B(v)$.\\
\end{center}
By $\mathbf{(RP)}$,  we can get
\begin{center}
$A^*(u)\odot((A^*(u)\rightarrow A(u))\odot(A(u)\rightarrow B(v)))\leq B(v)$.\\
\end{center}
It means that $B$ is an upper bound of the set
\begin{center}
$\{A^*(u)\odot(A^*(u)\rightarrow A(u))\odot(A(u)\rightarrow B(v))|u\in U, v\in V\}$.\\
\end{center}
Thus, $B^*(v)\leq B(v)$ for each $v\in V$.

On the other hand, if $A^*=A$ is a normal fuzzy set, that is, $\exists u_0\in U$ such that $A^*(u_0)=A(u_0)=1$, then
\begin{eqnarray*}B^*(v)&=&\sup\limits_{u\in U}\{A^*(u)\odot((A^*(u)\rightarrow A(u))\odot(A(u)\rightarrow B(v)))\}\\&\geq &A^*(u_0)\odot((A^*(u_0)\rightarrow A(u_0))\odot(A(u_0)\rightarrow B(v)))\\&=& 1\odot((1\rightarrow 1)\odot(1\rightarrow B(v)))\\&=& B(v).
\end{eqnarray*}
That is $B(v)\leq B^*(v)$.
\end{proof}

\begin{example}
\emph{Let us continue with Example \ref{4.1.2}. Now, if $A^*(u)=A(u)$, then}

\begin{center}
$A^*(u)\rightarrow A(u)=[1~~1~~1~~1]$.
\end{center}

\begin{center}
$(A^*(u)\rightarrow A(u))\odot (A(u)\rightarrow B(v))=A(u)\rightarrow B(v)$.\\
\end{center}
\emph{Hence, we can get}
\begin{eqnarray*}B^*(v)&=&\sup\limits_{u\in U}\{A^*(u)\odot((A^*(u)\rightarrow A(u))\odot(A(u)\rightarrow B(v)))\}\\&=& \sup\limits_{u\in U}\{A^*(u)\odot(A(u)\rightarrow B(v))\}\\&=&[0.2~~0.5~~0.9~~1].
\end{eqnarray*}
\emph{Therefore, $B^*(v)=B(v)$.}
\end{example}

\subsection{The quintuple implication principle of multiple-rules FMP problems }

In this subsection, we will apply the QIP method to multiple-rules fuzzy reasoning. Because in many practical applications, most control systems are composed of multiple-rules.

A general multiple-rules model is as follows:
\begin{center}
\begin{tabular}{l}
  \hline
  rule 1~~If $u_1$ is $A_{11}$ and $\cdots$ and $u_n$ is $A_{1n}$,~~then $v$ is $B_1$ \\

  rule 2~~If $u_1$ is $A_{21}$ and $\cdots$ and $u_n$ is $A_{2n}$,~~then $v$ is $B_2$ \\

  ~~~~~~~~~~~~~~~~~~~~~~~~~~~~~~~~$\vdots$ \\

  rule $m$~~If $u_1$ is $A_{m1}$ and $\cdots$ and $u_n$ is $A_{mn}$,~~then $v$ is $B_m$ \\

  input~~~ $u_1$ is $A^*_1$ ~~~~~and $\cdots$ and $u_n$ is $A^*_n$ \\

  output~~~~~~~~~~~~~~~~~~~~~~~~~~~~~~~~~~~~~~~~~~~~~~~~~~~~~~~~~~~~~~~~~$v$ is $B^*$. \\
  \hline\\
\end{tabular}\\
Model 1. General multiple-rules model
\end{center}
where $A_{ij}$, $A^*_j$ are fuzzy sets on  $U_j$, and $B_i$, $B^*$ are fuzzy sets on $V$, $i=1,2,\cdots,m$, $j=1,2,\cdots,n$.

Since the variable $u_1,u_2,\cdots,u_n$ can be considered as a multi-dimensional variable $u$, suppose that

\begin{center}
$u=(u_1,u_2,\cdots,u_n)\in U=U_1\times U_2\times\cdots\times U_n$.
\end{center}
Every rules can be looked as a fuzzy relationship on $U$, that is

\begin{center}
$A_i(u)=A_{i1}(u_1)\odot A_{i2}(u_2)\odot\cdots\odot A_{in}(u_n)$\\
\end{center}
Hence, the Model 1 can be simplified as Model 2:

\begin{center}
\begin{tabular}{l}
  \hline
  rule 1 ~~~~If $A_1(u)$, then $B_1(v)$ \\

  rule 2 ~~~~If $A_2(u)$, then $B_2(v)$ \\

  ~~~~~~~~~~~~~~~~~~~~~~~~~~ $\vdots$ \\

  rule $m$ ~~~If $A_m(u)$, then $B_m(v)$ \\

  input~~~~~~~~~~$A^*(u)$ \\

  output~~~~~~~~~~~~~~~~~~~~~~~~~~$B^*(v)$. \\
  \hline\\
\end{tabular}\\
Model 2\\
\end{center}

In fact, the variable $u$ in Model 2 can be as either a multi-dimensional variable or a one-dimensional variable. In the following, suppose that $u$ is a one-dimensional variable.

In fuzzy reasoning, the method often used is First Inference Then Aggregation(FITA). That is, first, the Model 2 is disassembled into the following FMP problems:

\begin{center}
\begin{tabular}{l}
  \hline
  rule~~~~If $A_i(u)$, then $B_i(v)$ \\

  input~~~~~~$A^*(u)$ \\

  output~~~~~~~~~~~~~~~~~~~$B^*_i(v)$.\\
  \hline\\
\end{tabular}
\end{center}
where $i=1,2,\cdots\cdots,n$. According to Theorem \ref{4.1.1}, the solution for every FMP problem is that
\begin{center}
$B^*_i(v)=\sup\limits_{u\in U}\{A^*(u)\odot[(A^*(u)\rightarrow A_i(u))\odot(A_i(u)\rightarrow B_i(v))]\}$.
\end{center}
for any $v\in V$. Then we can aggregate $B^*_i(v)$, $i=1,2,\cdots,n$, and get the solution of the Model 2:
\begin{eqnarray*}B^*(v)&=&\bigvee^m_{i=1}B^*_i(v)\\&=&\bigvee^m_{i=1}(\sup\limits_{u\in U}\{A^*(u)\odot((A^*(u)\rightarrow A_i(u))
\odot(A_i(u)\rightarrow B_i(v)))\}).
\end{eqnarray*}

The following example shows that reversibility of QIP method in Theorem \ref{4.1.3} is not applicable to multiple-rules FMP problems.

\begin{example}
\emph{Let $U=V=[0,1]$, $A_1, A_2, A^*\in \mathcal{F}(U)$, $B_1,B_2,B^*\in \mathcal{F}(V)$. Then we consider the following fuzzy rules:}

\begin{center}
\emph{rule 1: If $A_1(u)$, then $B_1(v)$}
\end{center}
\begin{center}
\emph{rule 2: If $A_2(u)$, then $B_2(v)$}\\
\end{center}
\emph{The correspond membership of the above fuzzy sets are given in the following:}

\begin{center}
~~$A_1=\{(u_1, 0.1), (u_2, 0.6), (u_3, 1)\}$,
\end{center}
\begin{center}
~$B_1=\{(v_1, 0.3), (v_2, 0.7), (v_3, 1)\}$,
\end{center}
\begin{center}
~~~$A_2=\{(u_1, 0.2), (u_2, 0.4), (u_3, 0.8)\}$,
\end{center}
\begin{center}
$B_2=\{(v_1, 0.1), (v_2, 0.8), (v_3, 1)\}$.\\
\end{center}
\emph{Suppose that $A^*=A_1$, we calculate the value of $B^*$. The semi-overlap function we choose and its corresponding residual implication are}

\begin{center}
$u\odot v=min\{u,v\}=u\wedge v$\\
\end{center}
\emph{and}

\begin{center}
$u\rightarrow v$=$\left\{
  \begin{array}{ll}
    1, & if~u\leq v; \\
    v, & if~u> v.
  \end{array}
\right.$\\
\end{center}
\emph{According to FITA, we can do it in three steps:}\\

\emph{(1) $B^*_1(v)=\sup\limits_{u\in U}\{A^*(u)\odot[(A^*(u)\rightarrow A_1(u))\odot(A_1(u)\rightarrow B_1(v))]\}=[0.3~~0.7~~1]$;}\\

\emph{(2) $B^*_2(v)=\sup\limits_{u\in U}\{A^*(u)\odot((A^*(u)\rightarrow A_2(u))\odot(A_2(u)\rightarrow B_2(v)))\}=[0.1~~0.8~~0.8]$;}\\

\emph{(3) $B^*(v)=B^*_1(v)\vee B^*_2(v)=[0.3~~0.8~~1]$.}\\
\emph{Hence, $B^*(v)\neq B(v)$.}
\end{example}

In the following, we will study the reducibility of multiple-rules, which is based on similarity measure. Therefore, we introduce the definition of similarity measure.

\begin{definition}(\cite{Bustince4})
\emph{A function $S: \mathcal{F}(U)\times \mathcal{F}(U)\rightarrow [0,1]$ is called a similarity measure on $\mathcal{F}(U)$, if $S$ has the following properties:}

\emph{$(1)$ $S(A,B)=S(B,A)$, for all $A,B\in \mathcal{F}(U)$;}

\emph{$(2)$ $S(A,A^c)=0$ if and only if $A$ is crisp;}

\emph{$(3)$ $S(A,B)=1$ if and only if $A=B$;}

\emph{$(4)$ For all $A,B,C,D\in \mathcal{F}(U)$, if $A\leq B\leq C\leq D$, then $S(A,B)\leq S(C,D)$;}

\emph{$(5)$ $S(A,B)=S(A^c,B^c)$.}
\end{definition}

\begin{example}(\cite{Bustince4})
\emph{Let $A,B\in \mathcal{F}(U)$. Then }

\begin{center}
  $S(A,B)=\frac{1}{n}\sum\limits_{i=1}^n1-|A(u_i)-B(u_i)|$.
\end{center}
\emph{is a similarity measure.}
\end{example}

Then, the fuzzy reasoning algorithm based on similarity measure is as follows:\\

\begin{tabular}{l}
  \hline
  $Algorithm~~ 1$ \\
  \hline
  $\bullet$~Input: A set of rules $R_i$ with $i\in\{1,\cdots, m\}$ and a fact $A^*_j$ with $j\in\{1,\cdots, n\}$ \\
  $\bullet$~Output: $B^*$. \\
  $\bullet$~According to $A^*_j, A^*_j\rightarrow A_{ij}$ and $A_{ij}\rightarrow B_i$, we can calculate that:\\
  ~~~~~~~~~~~~~~~~~~~~~~~~~~~~~$B^*_{ij}=\sup\limits_{x\in X}\{A^*_j\odot((A^*_j\rightarrow A_{ij})\odot (A_{ij}\rightarrow B_i))\}$.\\
  $\bullet$~Making the union of these $m$ results, \\
  ~~~~~~~~~~~~~~~~~~~~~~~~~~~~$B^*_i=\cup B^*_{ij}$ \\
  $\bullet$~Calculate the average similarity. Let $S_{ij}$ be the similarity measure between $A_{ij}$ and $A^*_j$, \\
  ~~~and the average similarity $S_i$ is:\\
  ~~~~~~~~~~~~~~~~~~~~~~~~~~~$S_i=\frac{1}{n}(S_{i1}+S_{i2}+\cdots +S_{in})$.\\
  $\bullet$~We give weight to rule $i$.\\
  ~~~~~~~~~~~~~~~~~~~~~~~~~~~$w_i=\left\{
  \begin{array}{ll}
    \frac{1-S}{2}, & i\notin I \\
    \frac{1+S}{2}, & i\in I.\\
  \end{array}
  \right.$\\
  ~~~where $S=max\{S_1, S_2, \cdots, S_n\}$ and $I=\{i|S_i=S, 1\leq i\leq n\}$.\\
  $\bullet$~Calculation reasoning results.\\
  ~~~~~~~~~~~~~~~~~~~~~~~~~~~~~$B^*=\cup w_iB^*_i$.\\
  \hline\\
 \end{tabular}

\begin{theorem}
\emph{Let $A_{ij}(i=1,2,\cdots,m; j=1,2,\cdots, n)$ be a normal fuzzy set, that is exist a $u_{ij}\in U_j$ such that $A_{ij}(u_{ij})=1$. Then when $A^*_j=A_{i_0j}$, we can get $B^*=B_{i_0}$ according to the above algorithm.}
\end{theorem}
\begin{proof}
Suppose that for $i_0$($1\leq i_0\leq m$),  $A^*_j=A_{i_0j}$, where $j=1,2,\cdots,n$. By Theorem 4.4, we can get $B^*_{i_0j}=B_{i_0}$, hence $B^*_{i_0}=\cup B^*_{i_0j}=B_{i_0}$. If $A^*_j=A_{i_0j}$, we have $S_{i_0j}(A_{i_0j},A^*_j)=1$.
 Therefore, $S_{i_0}=\frac{1}{n}(S_{i_01}+S_{i_02}+\cdots +S_{i_0n}) =1$. By algorithm 1,  $S=max\{S_1, S_2,\cdots,S_{i_0},\cdots,S_n\}$, we have $S=S_{i_0}=1$, thus, $w_{i_0}=1$, $w_i=0$, where $i=1,2\cdots,m$ and $i\neq i_0$. Hence $B^*=\cup w_iB^*_i=B_{i_0}$.
\end{proof}

\section{An application of semi-overlap functions and quintuple implication principle in classification problems}

In this section, we apply the combination of semi-overlap functions and quintuple implication principle to classification problems. Firstly, we introduce fuzzy association rule-based classification method for high-dimensional problems ($\mathbf{FARC-HD}$, for short). Then, based on semi-overlap functions and quintuple implication principle, we propose a new fuzzy classification algorithm  which called $\mathbf{SO5I-FRC}$. Finally, we apply the two methods to the classification problem respectively. And the experimental results show that the classification accuracy of SO5I-FRC algorithm is higher than that of FARC-HD algorithm.

\subsection{Fuzzy rule-based classification systems}

In the literature, there are multiple techniques used to solve classification problems. Among them, Fuzzy Rule-Based Classification Systems (FRBCSs, for short) are one of the most popular approaches, since they provide an interpretable model by means of the use of linguistic labels in their rules\cite{Ishibuchi}. The main purpose of rule-based classification algorithm (fuzzy or clear rules) is to classify the samples that satisfied  certain conditions in the data set. Here, we will introduce two fuzzy rule based classification algorithms.

$\bullet$ $\mathbf{FARC-HD}$

J.Alcal\'{a}-Fdes(\cite{AlcalaFdez}) propose a fuzzy association rule based classification method for high-dimensional problems(FARC-HD) to obtain an accurate and compact fuzzy rule-based classifier with a low computational cost. In this paper, we only consider the classification problem with two variables. Hence, we will give the rule structure and algorithm with two variables in FARC-HD. The rule structure is as follows:

\begin{center}
Rules $R_j$: If $u_1$ is $A_{j1}$ and  $u_2$ is $A_{j2}$,
\end{center}
\begin{center}
~~~~~~~~~~~~~~~~then Class is $C_j$ with $RW_j$\\
\end{center}
where $R_j(j=1,2,\cdots,m)$ is the label of the $j$-th rule, $u=(u_1,u_2)$ is a vector representing the example, $A_{j1}$ and $A_{j2}$ is a linguistic label modeled by a triangular membership function, $C_j$ is the class label and $RW_j$ is the rule weight computed using the certainty factor defined in \cite{Ishibuchi1}.

The FARC-HD algorithm is as follows:\\

\begin{tabular}{l}
  \hline
  $Algorithm~~2$ \\
  \hline
  $\bullet$~Input: A set of rules $R_j$ and rule weight $RW_j$ and a sample $u^*=(u_1, u_2)$  \\
  $\bullet$~Output: $C_j$. \\
  \hline\\
 \end{tabular}

\begin{tabular}{l}
  \hline
  $Algorithm~~2$ \\
  \hline
  $\bullet$~Matching degree. Calculate the matching degree between the antecedent for all rules \\
  ~~~and the samples $u^*$.\\
  ~~~~~~~~~~~~~~~~~~~~~~~~~$\mu_{A_j}(u^*)=O(\mu_{A_{j1}}(u^*_1),\mu_{A_{j2}}(u^*_2))$.\\
  $\bullet$~ Association degree. Calculate the association degree between the sample $u^*$ and \\
  ~~~each rule in the rule base.\\
  ~~~~~~~~~~~~~~~~~~~~~~~~~$b_j(u^*)=\mu_{A_j}(u^*)\cdot RW_j$\\
  $\bullet$~Confidence degree. Calculate the confidence degree of each class\\
  ~~~~~~~~~~~~~~~~~~~~~~~~~$conf_i(u^*)=\sum b_j(u^*)$\\
  $\bullet$~Classification. Predict the class with the highest confidence.\\
  ~~~~~~~~~~~~~~~~~~~~~~~~~Class=arg $max\{conf_i(u^*)\}$.\\
  \hline\\
 \end{tabular}

In algorithm 2, $O$ represents binary overlap functions. In fact, in \cite{Elkano}, based on FARC-HD algorithm, multi-class classification problems is studied by $n$-dimensional overlap functions, and good results are obtained. In this paper, we study the classification problem with two variables, hence we choose the binary overlap functions.

$\bullet$ $\mathbf{SO5I-FRC}$

Based on semi-overlap functions and QIP method introduced in this paper, we propose a new fuzzy classification algorithm, which called SO5I-FRC, for short. The main idea of SO5I-FRC algorithm is as follows: we apply the semi-overlap functions and their residual implications to the QIP method and calculate the matching degree between an input sample and each rules in fuzzy rule set, then the class of the rule with the highest matching degree is the classification result of this sample. The difference between SO5I-FRC algorithm and FARC-HD algorithm is that in FARC-HD algorithm, we only need to calculate the matching degree between input $A^*$ and the antecedent $A_{ji}$ of the rule, however, in SO5I-FRC algorithm, we need to calculate the matching degree between $A^*$ and the whole rule, including the matching degree of $A_{ji}$ and $B_j$. The rule structure of SO5I-FRC algorithm is as follows: (we only consider the classification problem with two variables)

\begin{center}
~~Rules $R_j$: If $u_1$ is $A_{j1}$ and $u_2$ is $A_{j2}$,
\end{center}
\begin{center}
then $C_j$ is $B_j$.\\
\end{center}
where $R_j(j=1,2,\cdots,m)$ is the $j$-th rule of the fuzzy rule set; $A_{j1}$ and $A_{j2}$ are two one-dimensional fuzzy sets of the antecedents of fuzzy rules; $C_j$ is the classification label of the sample satisfying fuzzy sets $A_{j1}$ and $A_{j2}$; $B_j$ is the degree to which $C_j$ belongs to class 1 or class 2, and it is a single point fuzzy set.

The SO5I-FRC algorithm is as follows:\\

\begin{tabular}{l}
  \hline
  $Algorithm~~3$ \\
  \hline
  $\bullet$~Input: A set of rules $R_j$ and a sample $u^*=(u_1, u_2)$  \\
  $\bullet$~Output: $C_j$. \\
  $\bullet$~Calculating two-dimensional fuzzy sets $A^*$ and $A_j$\\
  ~~~~~~~~~~~~~~~~~~~~$A^*(u_1,u_2)=min\{A^*_{1}(u_1), A^*_{2}(u_2)\}$;\\
  ~~~~~~~~~~~~~~~~~~~~$A_j(u_1,u_2)=min\{A_{j1}(u_1), A_{j2}(u_2)\}$.\\
  $\bullet$~Calculating the matching degree between the rules $R_j$ and the antecedent $A^*$\\
  ~~~~~~~~~~~~~~~~~~~~$B_j^*(v)=\sup\limits_{u\in U}\{A^*(u)\odot((A^*(u)\rightarrow A_j(u))\odot(A_j(u)\rightarrow B_j(v)))\}$\\
  $\bullet$~The classification result of $A^*$ is the class of the rule with the highest matching degree\\
  ~~~~~~~~~~~~~~~~~~~~Class=arg $max\{B_j^*(v)\}$.\\
  \hline
 \end{tabular}

\subsection{Experimental framework}

$\bullet$ $\mathbf{step~ 1}$. Selection of data sets

We will select the Banana data set in the KEEL data set repository(https://sci2s.ugr.es/keel/\\category.php?cat=cla), which is a binary classified data set with two numerical attributes, as shown in Table 2.

\begin{center}
\begin{tabular}{cccc}
  \hline
  number & attribute 1 & attribute 2 & class \\
  \hline
  $1$ & $1.14$   & $-0.114$  & $1$ \\
  $2$ & $-1.52$  & $-1.15$   & $2$ \\
  $3$ & $-1.05$  & $0.72$    & $1$ \\
  $4$ & $-0.916$ & $0.397$   & $2$ \\
  $5$ & $-1.09$  & $0.437$   & $2$ \\
  $6$ & $-0.584$ & $0.0937$  & $2$ \\
  $\cdots$ & $\cdots$ & $\cdots$ & $\cdots$ \\
  $5298$ & $1.57$   & $-0.389$  & $1$ \\
  $5299$ & $-0.411$ & $0.727$   & $2$ \\
  $5300$ & $-0.72$  & $0.189$   & $2$ \\
  \hline\\
\end{tabular}\\
Table 2. Banana data set
\end{center}

The distribution of the two classes is shown in Figure \ref{fig.5.1}. Here, the red is the data of the first class and the green is the data of the second class.

\begin{figure}[!ht]
  \centering
  \includegraphics[width=10.0cm,height=7.0cm]{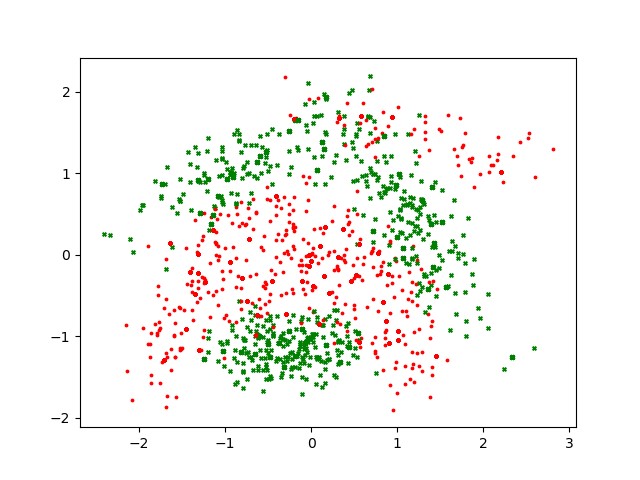}\\
  \caption{Classification distribution}\label{fig.5.1}
\end{figure}

$\bullet$ $\mathbf{step~ 2}$. Establish accurate rule set

According to RIPPER algorithm, the rule set (I) of the Banana data set is obtained:

(1) if $-2.41\leq u_1\leq -1.26$ and $-1.9\leq u_2\leq 0.52$, then class is 1;

(2) if $-0.9\leq u_1\leq 0.6$ and $0.52\leq u_2\leq 0.8$, then class is 1;

(3) if $-1.26 \leq u_1\leq 0.9$ and $-0.65 \leq u_2\leq 0.52$, then class is 1;

(4) if $0.9\leq u_1\leq 1.23$ and $-0.65\leq u_2\leq -0.15$, then class is 1;

(5) if $0.56\leq u_1\leq 1.6$ and $-1.9\leq u_2\leq -0.65$, then class is 1;

(6) if $1.3\leq u_1\leq 2.81$ and $0.81\leq u_2\leq 2.2$, then class is 1;

otherwise, class is 2.

According to the actual classification result and data set characteristics, we subdivide "otherwise, class is 2." in the above rule set (I) into the following 7 rules:

(7) if $-2.41\leq u_1\leq -0.9$ and $0.52\leq u_2\leq 2.2$, then class is 2;

(8) if $-0.9\leq u_1\leq 0.6$ and $0.8\leq u_2\leq 2.2$, then class is 2;

(9) if $0.6 \leq u_1\leq 1.3$ and $0.52\leq  u_2\leq 2.2$, then class is 2;

(10) if $0.9\leq u_1\leq 1.23$ and $-0.15\leq u_2\leq 0.52$, then class is 2;

(11) if $-1.26\leq u_1\leq 0.56$ and $-1.9\leq u_2\leq -0.65$, then class is 2;

(12) if $1.23\leq u_1\leq 2.81$ and $-0.65\leq u_2\leq 0.81$, then class is 2;

(13) if $1.6\leq u_1\leq 2.81$ and $-1.9\leq u_2\leq -0.65$, then class is 2.

Rule set (II) is composed of 6 rules in rule set (I) and 7 rules obtained by subdividing "otherwise, class is 2". Our experiment is based on the 13 rules in rule set (II), and applies SO5I-FRC algorithm and FARC-HD algorithm to carry out classification experiment, respectively.

$\bullet$ $\mathbf{step~ 3}$. Fuzzification methods of major premise (rules) and minor premise

Fuzzification is to convert the determined real number input into fuzzy quantity, that is, fuzzy set and corresponding membership function. Here, we need to fuzzify each rules in rule set (II) to obtain the fuzzy rule set (III).
When the condition of variable $u_i$ of rules $R_j$ in rule set (II) is $a \leq u_i \leq b$, the membership function in fuzzy rule set (III) is the trapezoidal membership function $T(a,b)$ shown in Fig 3. Fuzzy set $A_{ji}(i=1,2)$ is shown in the following, where the domain of $u_1$ is $[- 2.41, 2.81]$ and the domain of $u_2$ is $[- 1.9, 2.2]$.

\begin{center}
$A_{ji}(u)=\left\{
  \begin{array}{ll}
    1, & if~ a+0.1\leq u_i\leq b-0.1 \\
    \frac{2(u-a+0.25)}{7}, & if~ a-0.25\leq u_i\leq a+0.1\\
    \frac{-2(u-b-0.25)}{7}, & if~ b-0.1\leq u_i\leq a+0.25\\
    0, & otherwise
  \end{array}
\right.$
\end{center}
\begin{figure}[!ht]
  \centering
  \includegraphics[width=10.0cm,height=8.0cm]{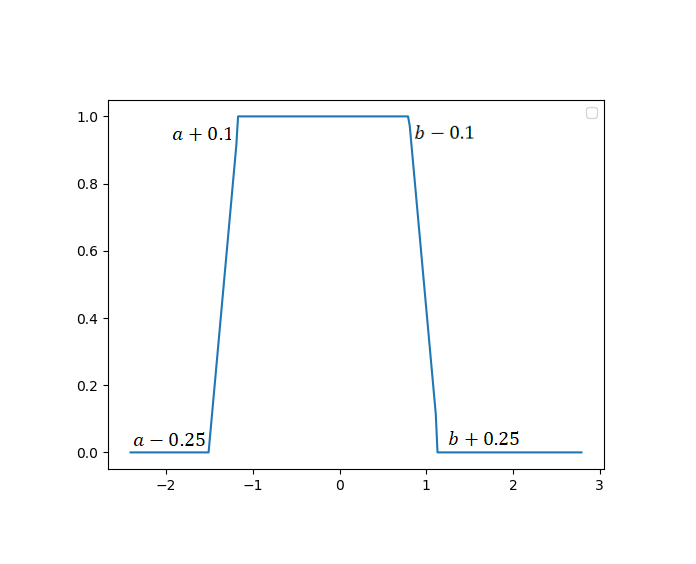}
  \caption{Trapezoidal membership function}\label{fig.3}
\end{figure}

The membership degree of the fuzzy rule consequent is the fuzzy set $B_j=\{a\}$,  the value of $a$ is the classification accuracy of all points in the area determined by variables $u_1$ and $u_2$ of $R_j$ in rule set (II). That is

\begin{center}
$B_1=\{0.891|C_j=1\}$,  $B_2=\{0.941|C_j=1\}$, $B_3=\{0.896|C_j=1\}$,
\end{center}
\begin{center}
$B_4=\{0.923|C_j=1\}$, $B_5=\{0.958|C_j=1\}$,  $B_6=\{0.637|C_j=1\}$,
\end{center}
\begin{center}
$B_7=\{0.996|C_j=2\}$,  $B_8=\{0.778|C_j=2\}$, $B_9=\{0.795|C_j=2\}$,
\end{center}
\begin{center}
~~$B_{10}=\{0.968|C_j=2\}$, $B_{11}=\{0.794|C_j=2\}$,  $B_{12}=\{0.987|C_j=2\}$
\end{center}
\begin{flushleft}
~~~~~~~~~~~~~~~~~~~~~~$B_{13}=\{0.983|C_j=2\}$.
\end{flushleft}

Next, we fuzzify the minor premise $A^*_{ji}$. The membership function expression of $A^*_{ji}$ and its membership function graph(Fig.4) are as follows

\begin{center}
$A^*_{ji}(u)=\left\{
  \begin{array}{ll}
    4u-4a+1, & if~ a-0.25\leq u_i\leq a \\
    -4u+4a+1, & if~ a\leq u_i\leq a+0.25 \\
    0, & otherwise
  \end{array}
\right.$
\end{center}
\begin{figure}[!ht]
  \centering
  \includegraphics[width=10.0cm,height=7.0cm]{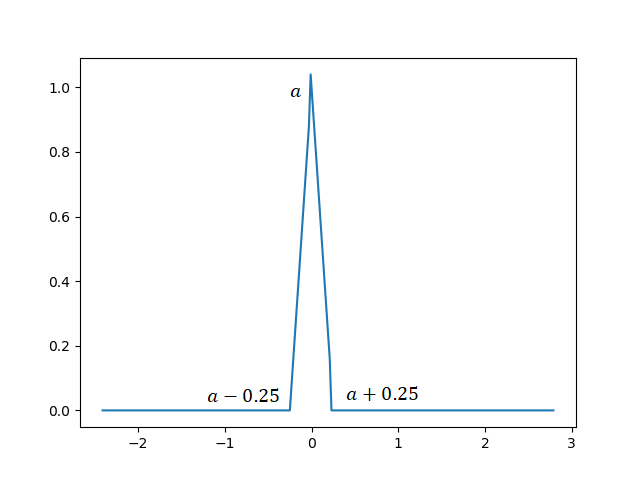}\\
  \caption{Triangular membership function}\label{fig.5}
\end{figure}

\subsection{Experimental results and analysis}

(1) In rule set (III),  we use FARC-HD algorithm to classify banana data set. We select the following  overlap functions for experiments and analyze the results.

\begin{center}
~~~~~~~~~~$O_a: u\odot_a v=min\{u,v\}$;
\end{center}
\begin{center}
$O_b: u\odot_b v=uv$;
\end{center}
\begin{center}
~~~~~$O_c: u\odot_c v=\frac{uv(u+v)}{2}$;
\end{center}
\begin{center}
~~~~~~~~~~~~~~$O_d: u\odot_d v=min\{\sqrt{u},\sqrt{v}\}$;
\end{center}
\begin{center}
~~~~~~~~$O_e: u\odot_e v=min\{u^2, v^2\}$.
\end{center}

In the following table, we will give the membership functions of variables $u_1$ and $u_2$ and the rule weight corresponding to each rules.

\begin{center}
\begin{tabular}{c|l|c}
  \hline
  rules & membership function of $u_1$/$u_2$    & rule weight \\
  \hline
  $1$  &  $Tra_{u_1}(-2.66,-2.31,-1.36,-1.01)$               & $1/6$\\
   ~   &  $Tra_{u_2}(-2.15,-1.8,0.42,0.77)$                  &~\\
  \hline
  $2$  & $Tra_{u_1}(-1.15,-0.8,0.5,0.85)$                    & $1/6$\\
   ~   & $Tra_{u_2}(0.27,0.62,0.7,1.05)$                     &~\\
  \hline
  $3$  & $Tra_{u_1}(-1.51,-1.16,0.8,1.15)$                   & $1/6$\\
  ~    & $Tra_{u_2}(-0.9,-0.55,0.42,0.77)$                   &~\\
  \hline
  $4$  & $Tra_{u_1}(0.65,1.0,1.13,1.48)$                     & $1/6$\\
  ~    & $Tra_{u_2}(-0.9,-0.55,-0.25,0.1)$                   &~\\
  \hline
  $5$  & $Tra_{u_1}(0.31,0.66,1.5,1.85)$                     & $1/6$\\
  ~    & $Tra_{u_2}(-2.15,-1.8,-0.75,-0.4)$                  &~\\
  \hline
  $6$  & $Tra_{u_1}(1.05,1.4,2.71,3.06)$                     & $1/6$\\
  ~    & $Tra_{u_2}(0.56,0.91,2.1,2.45)$                     &~\\
  \hline
  $7$  & $Tra_{u_1}(-2.66,-2.31,-1.0,-0.65)$                 & $1/7$\\
  ~    & $Tra_{u_2}(0.27,0.62,2.1,2.45)$                     &~\\
  \hline
  $8$  & $Tra_{u_1}(-1.15,-0.8,0.5,0.85)$                    & $1/7$\\
  ~    & $Tra_{u_2}(0.55,0.9,2.1,2.45)$                      &~\\
  \hline
  $9$  & $Tra_{u_1}(0.35,0.7,1.2,1.55)$                      & $1/7$\\
  ~    & $Tra_{u_2}(0.27,0.62,2.1,2.45)$                     &~\\
  \hline
  $10$ & $Tra_{u_1}(0.65,1.0,1.13,1.48)$                     & $1/7$\\
  ~    & $Tra_{u_2}(-0.4,-0.05,0.42,0.77)$                   &~\\
  \hline
  $11$ & $Tra_{u_1}(-1.51,-1.16,0.46,0.81)$                  & $1/7$\\
  ~    & $Tra_{u_2}(-2.15,-1.8,-0.75,-0.4)$                  &~\\
  \hline
  $12$ & $Tra_{u_1}(0.98,1.33,2.71,3.06)$                    & $1/7$\\
  ~    & $Tra_{u_2}(-0.9,-0.55,0.71,1.06)$                   &~\\
  \hline
  $13$ & $Tra_{u_1}(1.35,1.7,2.71,3.06)$                     & $1/7$\\
  ~    & $Tra_{u_2}(-2.15,-1.8,-0.75,-0.4)$                 &~\\
  \hline
  \end{tabular}\\
  $~$\\
  Table 3. Membership functions and rule weight
 \end{center}

In the following, we will give the classification accuracy results. The Table 4 shows that the classification accuracy obtained by selecting the same overlap function for 13 rules, and the Table 5 shows the classification accuracy obtained by selecting different overlap functions for 13 rules.

\begin{center}
\begin{tabular}{c|c|c}
  \hline
  experiments & overlap functions &  accuracy\\
  \hline
  Exp 1 & $O_a$ & $75.09\%$ \\
  \hline
  Exp 2 & $O_b$ & $74.25\%$ \\
  \hline
  Exp 3 & $O_c$ & $74.23\%$ \\
  \hline
  Exp 4 & $O_d$ & $72.94\%$ \\
  \hline
  Exp 5 & $O_e$ & $74.28\%$ \\
  \hline
  \end{tabular}\\
$~$\\
 Table 4. Classification accuracy results by selecting \\
  the same overlap function
 \end{center}

From the Table 4, we can get that using FARC-HD algorithm and selecting the same overlap function for 13 rules, the classification accuracy of the five groups experiments are 75.09\%, 74.25\%, 74.23\%, 72.94\% and 74.28\%, respectively, and the average accuracy is 74.158\%.

\begin{center}
\begin{tabular}{c|c|c|c|c|c}
  \hline
  rules & Exp 1 & Exp 2 & Exp 3 & Exp 4 & Exp 5\\
  \hline
  $1$  & $O_c$  & $O_c$  & $O_c$  & $O_a$  & $O_c$\\
  $2$  & $O_b$  & $O_b$  & $O_d$  & $O_e$  & $O_c$\\
  $3$  & $O_a$  & $O_c$  & $O_a$  & $O_b$  & $O_d$\\
  $4$  & $O_a$  & $O_a$  & $O_a$  & $O_a$  & $O_a$\\
  $5$  & $O_d$  & $O_e$  & $O_c$  & $O_a$  & $O_e$\\
  $6$  & $O_d$  & $O_d$  & $O_c$  & $O_d$  & $O_a$\\
  $7$  & $O_c$  & $O_a$  & $O_e$  & $O_a$  & $O_d$\\
  $8$  & $O_b$  & $O_d$  & $O_a$  & $O_e$  & $O_a$\\
  $9$  & $O_e$  & $O_c$  & $O_c$  & $O_c$  & $O_c$\\
  $10$ & $O_a$  & $O_a$  & $O_b$  & $O_a$  & $O_a$\\
  $11$ & $O_d$  & $O_b$  & $O_a$  & $O_b$  & $O_c$\\
  $12$ & $O_a$  & $O_a$  & $O_c$  & $O_c$  & $O_c$\\
  $13$ & $O_e$  & $O_d$  & $O_c$  & $O_d$  & $O_b$\\
  \hline
  accuracy   & $76.92\%$    & $75.06\%$    & $75.11\%$   & $74.08\%$   & $74.92\%$\\
  \hline
  \end{tabular}\\
  $~$\\
  Table 5. Classification accuracy results by selecting \\
  the different overlap functions
\end{center}

From the Table 5, we can get that using FARC-HD algorithm and selecting the different overlap functions for 13 rules, the classification accuracy of the five groups experiments are 76.92\%, 75.06\%, 75.11\%, 74.08\% and 74.92\%, respectively, and the average accuracy is 75.218\%. It is easy to see that the average accuracy of selecting different overlap functions for 13 rules is higher than that of selecting the same overlap function.

(2)  In rule set (III), we use SO5I-FRC algorithm to classify banana data set. When calculating the matching degree between $A^*$ and rule $R_j$, the following four groups of semi-overlap functions with good experimental performance can be arbitrarily selected for 13 rules, with a total of $\rm{C}_{13}^4$ choices. It should be noted that the implication and semi-overlap functions in each rules must be the same semi-overlap function and its corresponding implication.

$SO_1$:~~~~~~ $u\odot_1 v=\left\{
  \begin{array}{ll}
    0, & if~ u+v\leq 1 \\
    uv, & otherwise
  \end{array}\\
\right.$\\

~~~~~~~~~~~~~~~$u\rightarrow_1 v=\left\{
  \begin{array}{ll}
    1, & if~ u\leq v \\
    max\{1-u, \frac{v}{u}\}, & otherwise
  \end{array}\\
\right.$\\

$SO_2$:~~~~~~$u\odot_2 v=max\{\frac{u+v-1}{u+v-uv},0\}$\\

~~~~~~~~~~~~~~$u\rightarrow_2 v=\left\{
  \begin{array}{ll}
    1, & if~ u\leq v \\
    \frac{uv+1-u}{uv+1-v}, & otherwise
  \end{array}\\
\right.$\\

$SO_3$:~~~~~~$u\odot_3 v=min\{u,v\}$\\

~~~~~~~~~~~~~~$u\rightarrow_3 v=\left\{
  \begin{array}{ll}
    1, & if~ u\leq v \\
    v, & otherwise
  \end{array}\\
\right.$\\

$SO_4$:~~~~~~$u\odot_4 v=
\left\{
  \begin{array}{ll}
    \frac{1+min\{2u-1,2v-1\}max\{(2u-1)^2, (2v-1)^2\}}{2}, & u,v\in(0.5,1] \\
    min\{u,v\}, & otherwise
  \end{array}\\
\right.$\\

~~~~~~~~~~~~~$u\rightarrow_4 v=
\left\{
  \begin{array}{ll}
    min\{1, max\{\frac{\sqrt{2v-1}}{2\sqrt{2u-1}}, \frac{v}{2(2u-1)^2}\}+\frac{1}{2}\},& u\in(0.5,1],\\
    ~ & v\in[0.5,1]\\
    v, & v\in[0,0.5), u> v\\
    1, & u\in[0,0.5], u\leq v
  \end{array}\\
\right.$\\
$~$

In the following table, we will give the membership function of variables $u_1$ and $u_2$ in each rules.

\begin{center}
\begin{tabular}{c|l}
  \hline
  rules & membership function of $x_1$/$x_2$   \\
  \hline
  $1$  &  $Tra_{u_1}(-2.66,-2.31,-1.36,-1.01)$,~~~~ $Tra_{u_2}(-2.15,-1.8,0.42,0.77)$  \\
  \hline
  $2$  & $Tra_{u_1}(-1.15,-0.8,0.5,0.85)$,~~~~~~~~~~~~~ $Tra_{u_2}(0.27,0.62,0.7,1.05)$  \\
  \hline
  $3$  & $Tra_{u_1}(-1.51,-1.16,0.8,1.15)$,~~~~~~~~~~~ $Tra_{u_2}(-0.9,-0.55,0.42,0.77)$ \\
  \hline
  $4$  & $Tra_{u_1}(0.65,1.0,1.13,1.48)$,~~~~~~~~~~~~~~~~ $Tra_{u_2}(-0.9,-0.55,-0.25,0.1)$  \\
  \hline
  $5$  & $Tra_{u_1}(0.31,0.66,1.5,1.85)$,~~~~~~~~~~~~~~~~ $Tra_{u_2}(-2.15,-1.8,-0.75,-0.4)$ \\
  \hline
  $6$  & $Tra_{u_1}(1.05,1.4,2.71,3.06)$,~~~~~~~~~~~~~~~~ $Tra_{u_2}(0.56,0.91,2.1,2.45)$  \\
  \hline
   $7$  & $Tra_{u_1}(-2.66,-2.31,-1.0,-0.65)$,~~~~~~ $Tra_{u_2}(0.27,0.62,2.1,2.45)$ \\
  \hline
  $8$  & $Tra_{u_1}(-1.15,-0.8,0.5,0.85)$,~~~~~~~~~~~~~ $Tra_{u_2}(0.55,0.9,2.1,2.45)$  \\
  \hline
  $9$  & $Tra_{u_1}(0.35,0.7,1.2,1.55)$,~~~~~~~~~~~~~~~~~~ $Tra_{u_2}(0.27,0.62,2.1,2.45)$ \\
  \hline
   $10$ & $Tra_{u_1}(0.65,1.0,1.13,1.48)$,~~~~~~~~~~~~~~~~~$Tra_{u_2}(-0.4,-0.05,0.42,0.77)$ \\
  \hline
   $11$ & $Tra_{u_1}(-1.51,-1.16,0.46,0.81)$,~~~~~~~~ $Tra_{u_2}(-2.15,-1.8,-0.75,-0.4)$ \\
  \hline
   $12$ & $Tra_{u_1}(0.98,1.33,2.71,3.06)$,~~~~~~~~~~~~~ $Tra_{u_2}(-0.9,-0.55,0.71,1.06)$  \\
  \hline
  $13$ & $Tra_{u_1}(1.35,1.7,2.71,3.06)$,~~~~~~~~~~~~~~~ $Tra_{u_2}(-2.15,-1.8,-0.75,-0.4)$  \\
  \hline
  \end{tabular}\\
   $~$\\
  Table 6.  Membership functions
   \end{center}

The following table shows the results of five groups of experiments.
\begin{center}
\begin{tabular}{c|c|c|c|c|c}
  \hline
  rules & Exp 1 & Exp 2 & Exp 3 & Exp 4 & Exp 5\\
  \hline
  $1$  & $SO_3$  & $SO_3$  & $SO_1$  & $SO_3$  & $SO_3$\\
  $2$  & $SO_4$  & $SO_1$  & $SO_2$  & $SO_1$  & $SO_4$\\
  $3$  & $SO_2$  & $SO_1$  & $SO_3$  & $SO_3$  & $SO_3$\\
  $4$  & $SO_3$  & $SO_4$  & $SO_4$  & $SO_3$  & $SO_2$\\
  $5$  & $SO_1$  & $SO_3$  & $SO_1$  & $SO_1$  & $SO_1$\\
  $6$  & $SO_2$  & $SO_3$  & $SO_2$  & $SO_2$  & $SO_1$\\
  $7$  & $SO_3$  & $SO_4$  & $SO_3$  & $SO_4$  & $SO_3$\\
  $8$  & $SO_4$  & $SO_3$  & $SO_4$  & $SO_1$  & $SO_4$\\
  $9$  & $SO_2$  & $SO_2$  & $SO_1$  & $SO_4$  & $SO_2$\\
  $10$ & $SO_3$  & $SO_2$  & $SO_2$  & $SO_2$  & $SO_2$\\
  $11$ & $SO_2$  & $SO_1$  & $SO_3$  & $SO_1$  & $SO_2$\\
  $12$ & $SO_2$  & $SO_2$  & $SO_4$  & $SO_2$  & $SO_1$\\
  $13$ & $SO_2$  & $SO_2$  & $SO_3$  & $SO_4$  & $SO_1$\\
  \hline
  accuracy   & $86.58\%$    & $88.09\%$    & $88.09\%$   & $88.15\%$   & $86.58\%$\\
  \hline
  \end{tabular}\\
  $~$\\
  Table 7. Classification accuracy results
\end{center}

Using SO5I-FRC algorithm, we can get the classification accuracy of the five groups of experiments are 86.58\%, 88.09\%, 88.09\%, 88.15\% and 86.58\% respectively, and the average accuracy is 87.50\%.

The average accuracy of FARC-HD algorithm (using overlap functions but not quintuple implication principle) is 75.218\%, and the average accuracy of SO5I-FRC algorithm (using semi-overlap functions and quintuple implication principle) proposed in this paper is 87.50\%.  The experimental results show that the average accuracy of SO5I-FRC algorithm is higher than that of FARC-HD algorithm. On the other hand, we can conclude that semi-overlap functions and QIP method has a wide range of application and certain advantages in classification problems.

\section{Conclusions}

In this paper, we have studied semi-overlap functions and its corresponding quintuple implication principle and proposed a new fuzzy classification algorithm, which called SO5I-FRC algorithm. By doing comparative experiments, we can conclude that the average accuracy of SO5I-FRC algorithm is higher than FARC-HD algorithm. In the future work, we will try to use SO5I-FRC algorithm in more fields to get some better conclusions.

\medskip
\medskip
\noindent{\bf Acknowledgments.} This work is supported by the National Natural Science Foundation of China (61976130) and by the  Brazilian National Council for Scientific and Technological Development  -- CNPq (311429/2020-3).

~~\\

\begin{tabular}{|l|}
  \hline
  $\mathbf{Remark}$: The main results of this paper were introduced at a conference (The\\Seminar on Logics for New-Generation Artificial Intelligence, see: https://www.\\ xixilogic.org
  /events/ngl2021/ or https://rwsk.zju.edu.cn/2021/1027/c2075a243\\
  5645/page.htm ) in October 2021, but it has not been published anywhere.  \\
  \hline
\end{tabular}

\end{document}